\documentclass[12pt]{article}

\usepackage{amsmath, epsfig, cite}
\usepackage{amssymb}
\usepackage{amsfonts}
\usepackage{latexsym}
\usepackage{amsthm}

\newtheorem{thm}{Theorem}[section]

\newtheorem{lem}[thm]{Lemma}

\numberwithin{equation}{section}

\renewcommand{\thefootnote}

\begin{document}

\begin{center}
{\large\bf On some conjectural series containing harmonic numbers
\\of 3-order
 \footnote{Corresponding author$^*$. Email addresses: weichuanan78@163.com (C. Wei),
 cexu2020@ahnu.edu.cn (C. Xu).}}
\end{center}

\renewcommand{\thefootnote}{$\dagger$}

\vskip 2mm \centerline{Chuanan Wei$^1$, Ce Xu$^{2*}$}
\begin{center}
{$^1$School of Biomedical Information and Engineering\\
  Hainan Medical University, Haikou 571199, China\\
  $^2$School of Mathematics and Statistics\\
Anhui Normal University, Wuhu 241002, China}
\end{center}


\vskip 0.7cm \noindent{\bf Abstract.} Harmonic numbers are important
in a lot of branches of number theory. By means of the derivative
operator, the integral operator, and several summation and
transformation formulas for hypergeometric series, we prove four
series containing harmonic numbers of 3-order. Three of them are
conjectures which were recently proposed by Z.-W. Sun.



\section{Introduction}

For $\ell,n\in \mathbb{Z}^{+}$, define generalized harmonic numbers
of $\ell$-order as
\[H_{n}^{(\ell)}(x)=\sum_{k=1}^n\frac{1}{(x+k)^{\ell}}.\]
The $x=0$ case of them are
 harmonic numbers of $\ell$-order
\[H_{n}^{(\ell)}=\sum_{k=1}^n\frac{1}{k^{\ell}}.\]
When $\ell=1$, they reduce to classical harmonic numbers:
\[
H_{n}=\sum_{k=1}^n\frac{1}{k}.\] For a nonnegative integer $m$,
define the shifted-factorial to be
\begin{align*}
(x)_0=1 \quad\text{and}\quad (x)_m=x(x+1)\cdots(x+m-1) \quad
\text{when} \quad m\in \mathbb{Z}^{+}.
\end{align*}
 For a differentiable function $f(x)$, define the derivative operator
$\mathcal{D}_x$ by
\begin{align*}
\mathcal{D}_xf(x)=\frac{d}{dx}f(x)=\lim_{\bigtriangleup
x\to0}\frac{f(x+\bigtriangleup x)-f(x)}{\bigtriangleup x}.
 \end{align*}
 Then it is ordinary to show that
\begin{align*}
\mathcal{D}_x\:(1+x)_r=(1+x)_rH_r(x).
 \end{align*}
Several nice harmonic number identities from differentiation of the
shifted-factorials can be viewed in the papers
\cite{Liu,Paule,Sofo,Wang-Wei}.

Define the digamma function $\psi(x)$ as
\begin{align*}
\psi(x)=\frac{d}{dx}\big\{\log\Gamma(x)\big\},
\end{align*}
where $\Gamma(x)$ is the familiar gamma function. Furthermore, we
can define the polygamma function $\psi_n(x)$ to be
\begin{align*}
\psi_n(x)=\frac{d\,^{n+1}}{d\,x^{n+1}}\big\{\log\Gamma(x)\big\}=\frac{d\,^{n}}{d\,x^{n}}\psi(x).
 \end{align*}
For $s\in \mathbb{Z}^{+}$, a famous relation can be written as
\begin{align}
&\psi_s(z)=(-1)^{s+1}(1)_s\sum_{k=0}^{\infty}\frac{1}{(z+k)^{s+1}}.
 \label{relation}
\end{align}
 Three related special values of $\psi_2(x)$ (cf. \cite{Weisstein})
read
\begin{align}
&\psi_2(1)=-2\zeta_3,
 \label{polygamma-a}\\[2mm]
&\psi_2(\tfrac{1}{2})=-14\zeta_3,
 \label{polygamma-b}
\\[2mm]
&\psi_2(\tfrac{3}{4})=2\pi^3-56\zeta_3.
\label{polygamma-c}
\end{align}

There exist many interesting $\pi$-formulas in the literature. A
simple series for $1/\pi$ due to Bauer \cite{Bauer} and two fast
convergent series for $1/\pi^2$ due to Guillera
\cite{Guillera-a,Guillera-c} can be laid out as follows:
\begin{align}
&\qquad\sum_{k=0}^{\infty}(-1)^k(4k+1)\frac{(\frac{1}{2})_k^3}{(1)_k^3}
=\frac{2}{\pi},
\label{Bauer}\\[1mm]
&\quad\sum_{k=0}^{\infty}(20k^2+8k+1)\frac{\binom{2k}{k}^5}{(-2^{12})^k}=\frac{8}{\pi^2},
\label{Guillera-a}\\[1mm]
&\sum_{k=0}^{\infty}(820k^2+180k+13)\frac{\binom{2k}{k}^5}{(-2^{20})^k}=\frac{128}{\pi^2}.
\label{Guillera-b}
\end{align}
 In 2017,  He \cite{He} discovered the following
formula: for any prime $p\equiv 1\pmod 4$,
\begin{align}
&\sum_{k=0}^{(p-1)/2}\frac{6k+1}{4^k}\frac{(\frac{1}{2})_k^3(\frac{1}{4})_k}{(1)_k^4}\equiv
(-1)^{\frac{p+3}{4}}p\Gamma_p(\tfrac{1}{2})\Gamma_p(\tfrac{1}{4})^2
\pmod{p^2}, \label{He}
\end{align}
where $\Gamma_p(x)$ denotes the $p$-adic gamma function. For some
conclusions similar to \eqref{Bauer}-\eqref{He}, the reader is
referred to
 the papers \cite{Au,Guo,Guo-a,Liu-a,Liu-b,Liu-c,Mao,Sun-11,Sun-a,Wang,Wang-Sun}.

Recently, Sun \cite[Equation (3.50)]{Sun-c} proposed the following
conjecture containing harmonic numbers of 3-order associated with
\eqref{Bauer}.

\begin{thm}\label{thm-a}
\begin{align}
\sum_{k=1}^{\infty}(-1)^k(4k+1)\frac{(\frac{1}{2})_k^3}{(1)_k^3}H_{2k}^{(3)}=\frac{15\zeta(3)}{4\pi}-2G,
 \label{equation-wei-a}
\end{align}
where $\zeta(3)$ stands for
$$\zeta(3)=\sum_{k=0}^{\infty}\frac{1}{k^3}$$
and $G$ is the Catalan constant
$$G=\sum_{k=1}^{\infty}\frac{(-1)^{k-1}}{(2k-1)^2}.$$
\end{thm}

Similarly, we shall establish the following series  containing
harmonic numbers of 3-order related to \eqref{He}.

\begin{thm}\label{thm-b}
\begin{align}
\sum_{k=1}^{\infty}\frac{6k+1}{4^k}\frac{(\frac{1}{2})_k^3(\frac{1}{4})_k}{(1)_k^4}\Big\{64H_{2k}^{(3)}-7H_{k}^{(3)}\Big\}
=\frac{\sqrt{2}\,\Gamma(\frac{1}{4})^2}{\pi^{5/2}}\Big\{29\zeta(3)-\pi^3\Big\}.
 \label{equation-wei-b}
\end{align}
\end{thm}

Very recently,  Sun \cite[Equations (4.7) and (4.8)]{Sun-d} proposed
the following  two conjectures containing harmonic numbers of
3-order associated with \eqref{Guillera-a} and \eqref{Guillera-b}.

\begin{thm}\label{thm-c}
\begin{align}
\sum_{k=0}^{\infty}\bigg(\frac{-1}{4}\bigg)^k\frac{(\frac{1}{2})_k^5}{(1)_k^5}\bigg\{(20k^2+8k+1)H_{k}^{(3)}+\frac{8}{2k+1}\bigg\}=\frac{64\zeta(3)}{\pi^2}.
 \label{equation-wei-c}
\end{align}
\end{thm}

\begin{thm}\label{thm-d}
\begin{align}
\sum_{k=0}^{\infty}\bigg(\frac{-1}{1024}\bigg)^k\frac{(\frac{1}{2})_k^5}{(1)_k^5}\bigg\{(820k^2+180k+13)\Big[9H_{2k}^{(3)}-H_{k}^{(3)}\Big]+\frac{125}{2k+1}\bigg\}
=\frac{1024\zeta(3)}{\pi^2}.
 \label{equation-wei-d}
\end{align}
\end{thm}

The rest of the paper is organized as follows. According to the
operator methods and several summation and transformation formulas
for hypergeometric series, we shall certify Theorems \ref{thm-a} and
\ref{thm-b} in Sections 2. Similarly, Theorems \ref{thm-c} and
\ref{thm-d} will be verified in Section 3.

\section{Proof of Theorems \ref{thm-a} and \ref{thm-b}}
For the aim to prove Theorem \ref{thm-a}, we need the following
lemma.

\begin{lem}\label{lemm-a}
\begin{align}
\sum_{k=1}^{\infty}\frac{(\frac{1}{2})_k}{k(1)_k}\Big\{4H_{2k}^{(2)}-H_{k}^{(2)}\Big\}
=8\pi G-14\zeta(3).
 \label{equation-wei-e}
\end{align}
\end{lem}

\begin{proof}
Recall a formula due to Sun \cite[Theorem 1.1]{Sun-c}:
\begin{align*}
\frac{4(\arcsin\frac{x}{2})^2}{\sqrt{4-x^2}}=\sum_{k=1}^{\infty}\bigg(\frac{x}{2}\bigg)^{2k}\frac{(\frac{1}{2})_k}{(1)_k}\Big\{4H_{2k}^{(2)}-H_{k}^{(2)}\Big\}.
\end{align*}
Divide both sides by $x$ to deduce
\begin{align*}
\frac{4(\arcsin\frac{x}{2})^2}{x\sqrt{4-x^2}}=\sum_{k=1}^{\infty}\frac{x^{2k-1}}{4^k}\frac{(\frac{1}{2})_k}{(1)_k}\Big\{4H_{2k}^{(2)}-H_{k}^{(2)}\Big\}.
\end{align*}
Applying the integral operator $\int_{0}^2f(x)dx$ on both sides of
it, there is
\begin{align}
\int_{0}^2\frac{8(\arcsin\frac{x}{2})^2}{x\sqrt{4-x^2}}dx=\sum_{k=1}^{\infty}\frac{(\frac{1}{2})_k}{k(1)_k}\Big\{4H_{2k}^{(2)}-H_{k}^{(2)}\Big\}.
 \label{equation-wei-g}
\end{align}

Noting that
\begin{align*}
\log\bigg(\tan\frac{x}{2}\bigg)=-2\sum_{k=1}^{\infty}\frac{\cos\{(2k-1)x\}}{2k-1},
\end{align*}
we can proceed as follows:
\begin{align}
&\int_{0}^2\frac{8(\arcsin\frac{x}{2})^2}{x\sqrt{4-x^2}}dx
=\int_{0}^{\frac{\pi}{2}}\frac{4t^2}{\sin
t}dt=-8\int_{0}^{\frac{\pi}{2}}t\log\bigg(\tan\frac{t}{2}\bigg)dt
\notag\\
&\quad=16\sum_{k=1}^{\infty}\frac{1}{2k-1}\int_{0}^{\frac{\pi}{2}}t\cos\{(2k-1)t\}dt
\notag\\
&\quad=8\pi\sum_{k=1}^{\infty}\frac{(-1)^{k-1}}{(2k-1)^2}-16\sum_{k=1}^{\infty}\frac{1}{(2k-1)^3}
\notag\\
&\quad=8\pi G-14\zeta(3).
\label{equation-wei-h}
\end{align}
The combination of \eqref{equation-wei-g} and \eqref{equation-wei-h}
produces  \eqref{equation-wei-e}.
\end{proof}

Now we are ready to prove Theorem \ref{thm-a}.

\begin{proof}[Proof of Theorem \ref{thm-a}]
Whipple's transformation formula (cf.\cite[P. 32]{Bailey}) can be
expressed as
\begin{align}
&{_{7}F_{6}}\left[\begin{array}{cccccccc}
  a,1+\frac{a}{2},b,c,d,e,-n\\
  \frac{a}{2},1+a-b,1+a-c,1+a-d,1+a-e,1+a+n
\end{array};1
\right]
\notag\\
&=\frac{(1+a)_n(1+a-d-e)_n}{(1+a-d)_n(1+a-e)_n}
{_{4}F_{3}}\left[\begin{array}{cccccccc}
  1+a-b-c,d,e,-n\\
  1+a-b,1+a-c,d+e-a-n
\end{array};1
\right],
  \label{Whipple}
\end{align}
where the hypergeometric series has been defined by
$$
_{r+1}F_{r}\left[\begin{array}{c}
a_1,a_2,\ldots,a_{r+1}\\
b_1,b_2,\ldots,b_{r}
\end{array};\, z
\right] =\sum_{k=0}^{\infty}\frac{(a_1)_k(a_2)_k\cdots(a_{r})_k}
{(1)_k(b_1)_k\cdots(b_{r})_k}z^k.
$$
The $b=1,d\mapsto d-c$ case of \eqref{Whipple} reads
\begin{align*}
&\sum_{k=0}^{n}\frac{a+2k}{a}\frac{(c)_k(d-c)_k(e)_k(-n)_k}{(1+a-c)_k(1+a+c-d)_k(1+a-e)_k(1+a+n)_k}
\notag\\[1mm]
&\:\:= \frac{(1+a)_n(1+a+c-d-e)_n}{(1+a+c-d)_n(1+a-e)_n}
\sum_{k=0}^{n}\frac{a-c}{a-c+k}\frac{(d-c)_k(e)_k(-n)_k}{(1)_k(a)_k(d+e-a-c-n)_k}.
\end{align*}
Employ the derivative operator  $\mathcal{D}_{c}$ on both sides of
the last equation to get
\begin{align}
&\sum_{k=1}^{n}\frac{a+2k}{a}\frac{(c)_k(d-c)_k(e)_k(-n)_k}{(1+a-c)_k(1+a+c-d)_k(1+a-e)_k(1+a+n)_k}
\notag\\[1mm]
&\quad\times
\big\{H_{k}(c-1)-H_{k}(d-c-1)+H_{k}(a-c)-H_{k}(a+c-d)\big\}
\notag\\[1mm]
&\:\:=
\frac{(1+a)_n(1+a+c-d-e)_n}{(1+a+c-d)_n(1+a-e)_n}\big\{H_{n}(a+c-d-e)-H_{n}(a+c-d)\big\}
\notag\\[1mm]
&\quad\times
\sum_{k=0}^{n}\frac{a-c}{a-c+k}\frac{(d-c)_k(e)_k(-n)_k}{(1)_k(a)_k(d+e-a-c-n)_k}
\notag\\[1mm]
&\quad+
\frac{(1+a)_n(1+a+c-d-e)_n}{(1+a+c-d)_n(1+a-e)_n}\sum_{k=1}^{n}\frac{a-c}{a-c+k}\frac{(d-c)_k(e)_k(-n)_k}{(1)_k(a)_k(d+e-a-c-n)_k}
\notag\\[1mm]
&\quad\times
\bigg\{H_{k}(d+e-a-c-n-1)-H_{k}(d-c-1)-\frac{k}{(a-c)(a-c+k)}\bigg\}.
 \label{Whipple-a}
\end{align}
 The
$a=c=\frac{1}{2}$ case of \eqref{Whipple-a} gives
\begin{align*}
&\sum_{k=0}^{n}(4k+1)\frac{(\frac{1}{2})_k(d-\frac{1}{2})_k(e)_k(-n)_k}{(1)_k(2-d)_k(\frac{3}{2}-e)_k(\frac{3}{2}+n)_k}
\big\{H_{k}(-\tfrac{1}{2})-H_{k}(d-\tfrac{3}{2})+H_{k}-H_{k}(1-d)\big\}
\notag\\[1mm]
&\:\:= \frac{(\frac{3}{2})_n(2-d-e)_n}{(2-d)_n(\frac{3}{2}-e)_n}
\bigg\{H_{n}(1-d-e)-H_{n}(1-d)-\sum_{k=1}^{n}\frac{(d-\frac{1}{2})_k(e)_k(-n)_k}{k(1)_k(\frac{1}{2})_k(d+e-1-n)_k}
\bigg\}.
\end{align*}
Dividing both sides by $1-d$ and replacing $e$ by $\frac{3}{2}-d$,
we have
\begin{align}
&\sum_{k=1}^{n}(4k+1)\frac{(\frac{1}{2})_k(d-\frac{1}{2})_k(\frac{3}{2}-d)_k(-n)_k}{(1)_k(2-d)_k(d)_k(\frac{3}{2}+n)_k}
\bigg\{\sum_{i=1}^{k}\frac{1}{i(1-d+i)}-\sum_{i=1}^{k}\frac{1}{(-\frac{1}{2}+i)(d-\frac{3}{2}+i)}\bigg\}
\notag\\[1mm]
&\:\:= \frac{(\frac{3}{2})_n(\frac{1}{2})_n}{(1-d)_{n+1}(d)_n}
\bigg\{H_{n}(-\tfrac{1}{2})-H_{n}(1-d)-\sum_{k=1}^{n}\frac{(d-\frac{1}{2})_k(\frac{3}{2}-d)_k(-n)_k}{k(1)_k(\frac{1}{2})_k(\frac{1}{2}-n)_k}
\bigg\}.
  \label{Whipple-c}
\end{align}

Apply the derivative operator  $\mathcal{D}_{d}$ on both sides of
\eqref{Whipple-c} to find
\begin{align}
&\sum_{k=1}^{n}(4k+1)\frac{(\frac{1}{2})_k(d-\frac{1}{2})_k(\frac{3}{2}-d)_k(-n)_k}{(1)_k(2-d)_k(d)_k(\frac{3}{2}+n)_k}
\notag\\[1mm]
&\quad\times\bigg\{\big[H_{k}(d-\tfrac{3}{2})-H_{k}(\tfrac{1}{2}-d)+H_{k}(1-d)-H_{k}(d-1)\big]
\notag\\[1mm]
&\qquad\times\bigg[\sum_{i=1}^{k}\frac{1}{i(1-d+i)}-\sum_{i=1}^{k}\frac{1}{(-\frac{1}{2}+i)(d-\frac{3}{2}+i)}\bigg]
\notag\\[1mm]
&\qquad+\bigg[\sum_{i=1}^{k}\frac{1}{i(1-d+i)^2}+\sum_{i=1}^{k}\frac{1}{(-\frac{1}{2}+i)(d-\frac{3}{2}+i)^2}\bigg]
\bigg\}
\notag\\[1mm]
&\:\:=
\frac{(\frac{3}{2})_n(\frac{1}{2})_n}{(1-d)(2-d)_{n}(d)_n}\bigg[\frac{1}{1-d}+H_{n}(1-d)-H_{n}(d-1)\bigg]
\notag\\[1mm]
&\quad\times
\bigg\{H_{n}(-\tfrac{1}{2})-H_{n}(1-d)-\sum_{k=1}^{n}\frac{(d-\frac{1}{2})_k(\tfrac{3}{2}-d)_k(-n)_k}{k(1)_k(\frac{1}{2})_k(\tfrac{1}{2}-n)_k}
\bigg\}
\notag\\[1mm]
&\quad- \frac{(\frac{3}{2})_n(\frac{1}{2})_n}{(1-d)(2-d)_{n}(d)_n}
\notag\\[1mm]
&\quad\times
\bigg\{H_{n}^{(2)}(1-d)+\sum_{k=1}^{n}\frac{(d-\frac{1}{2})_k(\tfrac{3}{2}-d)_k(-n)_k}{k(1)_k(\frac{1}{2})_k(\tfrac{1}{2}-n)_k}
\big[H_{k}(d-\tfrac{3}{2})-H_{k}(\tfrac{1}{2}-d)\big] \bigg\}.
\label{Whipple-d}
\end{align}
The $2a=2c=d$ case of \eqref{Whipple-a} is
\begin{align}
\sum_{k=1}^{n}\frac{(e)_k(-n)_k}{k(1)_k(e-n)_k}= H_{n}(-e)-H_{n}.
 \label{Whipple-e}
\end{align}
Letting $d\to 1$ in \eqref{Whipple-d} and using L'H\^{o}pital rule
and \eqref{Whipple-e}, we arrive at
\begin{align*}
&\sum_{k=1}^{n}(4k+1)\frac{(\frac{1}{2})_k^3(-n)_k}{(1)_k^3(\frac{3}{2}+n)_k}H_{2k}^{(3)}
\notag\\[1mm]
&\:\:= \frac{(\frac{3}{2})_n(\frac{1}{2})_n}{8(1)_{n}^2}
\bigg\{H_{n}^{(3)}-\sum_{k=1}^{n}\frac{(\frac{1}{2})_k(-n)_k}{k(1)_k(\tfrac{1}{2}-n)_k}
\big[4H_{2k}^{(2)}-H_{k}^{(2)}\big]\bigg\}.
\end{align*}
Letting $n\to\infty$  and exploiting Lemma \ref{lemm-a}, we catch
hold of \eqref{equation-wei-a}.
\end{proof}

Subsequently, we shall prove Theorem \ref{thm-b}.

\begin{proof}[Proof of Theorem \ref{thm-b}]
 Remember the following summation formula for hypergeometric series due to Gosper (1977) (cf. \cite[Equation
 (5.1e)]{Chu}):
\begin{align}
&{_7}F_{6}\left[\begin{array}{c}
a-\frac{1}{2},\frac{2a+2}{3},2b-1,2c-1,2+2a-2b-2c,a+n,-n\\[3pt]
\frac{2a-1}{3},1+a-b,1+a-c,b+c-\frac{1}{2},2a+2n,-2n
\end{array};\, 1 \right]
\notag\\[2mm]
&\quad=\frac{(a+\frac{1}{2})_n(b)_n(c)_n(a-b-c+\frac{3}{2})_n}
{(\frac{1}{2})_n(1+a-b)_n(1+a-c)_n(b+c-\frac{1}{2})_n}.
\label{equation-a}
\end{align}
Employ the derivative operator  $\mathcal{D}_{b}$ on both sides of
\eqref{equation-a} to obtain
\begin{align*}
&\sum_{k=1}^n
\frac{(a-\frac{1}{2})_k(\frac{2a+2}{3})_k(2b-1)_k(2c-1)_k(2+2a-2b-2c)_k(a+n)_k(-n)_k}
{(1)_k(\frac{2a-1}{3})_k(1+a-b)_k(1+a-c)_k(b+c-\frac{1}{2})_k(2a+2n)_k(-2n)_k}
\notag\\[2pt]
&\quad\times\Big\{2H_k(2b-2)-2H_k(1+2a-2b-2c)+H_k(a-b)-H_k(b+c-\tfrac{3}{2})\Big\}
\notag\\[2pt]
&=\frac{(a+\frac{1}{2})_n(b)_n(c)_n(a-b-c+\frac{3}{2})_n}
{(\frac{1}{2})_n(1+a-b)_n(1+a-c)_n(b+c-\frac{1}{2})_n}
\notag\\[2pt]
&\quad\times\Big\{H_n(b-1)-H_n(a-b-c+\tfrac{1}{2})+H_n(a-b)-H_n(b+c-\tfrac{3}{2})\Big\}.
\end{align*}
Dividing both sides by $\frac{3}{2}+a-2b-c$, there holds
\begin{align}
&\sum_{k=1}^n
\frac{(a-\frac{1}{2})_k(\frac{2a+2}{3})_k(2b-1)_k(2c-1)_k(2+2a-2b-2c)_k(a+n)_k(-n)_k}
{(1)_k(\frac{2a-1}{3})_k(1+a-b)_k(1+a-c)_k(b+c-\frac{1}{2})_k(2a+2n)_k(-2n)_k}
\notag\\[2pt]
&\quad\times\bigg\{4\sum_{i=1}^k\frac{1}{(2b-2+i)(1+2a-2b-2c+i)}-\sum_{i=1}^k\frac{1}{(a-b+i)(b+c-\tfrac{3}{2}+i)}\bigg\}
\notag\\[2pt]
&=\frac{(a+\frac{1}{2})_n(b)_n(c)_n(a-b-c+\frac{3}{2})_n}
{(\frac{1}{2})_n(1+a-b)_n(1+a-c)_n(b+c-\frac{1}{2})_n}
\notag\\[2pt]
&\quad\times\bigg\{\sum_{i=1}^n\frac{1}{(b-1+i)(a-b-c+\tfrac{1}{2}+i)}-\sum_{i=1}^n\frac{1}{(a-b+i)(b+c-\tfrac{3}{2}+i)}\bigg\}.
\label{equation-b}
\end{align}

Apply the derivative operator  $\mathcal{D}_{c}$ on both sides of
\eqref{equation-b} to conclude
\begin{align*}
&\sum_{k=1}^n
\frac{(a-\frac{1}{2})_k(\frac{2a+2}{3})_k(2b-1)_k(2c-1)_k(2+2a-2b-2c)_k(a+n)_k(-n)_k}
{(1)_k(\frac{2a-1}{3})_k(1+a-b)_k(1+a-c)_k(b+c-\frac{1}{2})_k(2a+2n)_k(-2n)_k}
\notag\\[2pt]
&\quad\times\bigg\{\big[2H_k(2c-2)-2H_k(1+2a-2b-2c)+H_k(a-c)-H_k(b+c-\tfrac{3}{2})\big]
\notag\\[2pt]
&\qquad\times
\bigg[4\sum_{i=1}^k\frac{1}{(2b-2+i)(1+2a-2b-2c+i)}-\sum_{i=1}^k\frac{1}{(a-b+i)(b+c-\tfrac{3}{2}+i)}\bigg]
\end{align*}
\begin{align*}
&\qquad+
\bigg[8\sum_{i=1}^k\frac{1}{(2b-2+i)(1+2a-2b-2c+i)^2}+\sum_{i=1}^k\frac{1}{(a-b+i)(b+c-\tfrac{3}{2}+i)^2}\bigg]\bigg\}
\notag\\[2pt]
&=\frac{(a+\frac{1}{2})_n(b)_n(c)_n(a-b-c+\frac{3}{2})_n}
{(\frac{1}{2})_n(1+a-b)_n(1+a-c)_n(b+c-\frac{1}{2})_n}
\notag\\[2pt]
&\quad\times\bigg\{\Big[H_n(c-1)-H_n(a-b-c+\tfrac{1}{2})+H_n(a-c)-H_n(b+c-\tfrac{3}{2})\Big]
\notag\\[2pt]
&\qquad\times
\bigg[\sum_{i=1}^n\frac{1}{(b-1+i)(a-b-c+\tfrac{1}{2}+i)}-\sum_{i=1}^n\frac{1}{(a-b+i)(b+c-\tfrac{3}{2}+i)}\bigg]
\notag\\[2pt]
&\qquad+
\bigg[\sum_{i=1}^n\frac{1}{(b-1+i)(a-b-c+\tfrac{1}{2}+i)^2}+\sum_{i=1}^n\frac{1}{(a-b+i)(b+c-\tfrac{3}{2}+i)^2}\bigg]\bigg\}.
\end{align*}
The $a=b=c=\frac{3}{4}$ case of it engenders
\begin{align}
&\sum_{k=1}^{n}(6k+1)\frac{(\frac{1}{2})_k^3(\frac{1}{4})_k(\frac{3}{4}+n)_k(-n)_k}{(1)_k^4(\frac{3}{2}+2n)_k(-2n)_k}\Big\{64H_{2k}^{(3)}-7H_{k}^{(3)}\Big\}
\notag\\[1mm]
&\:\:=
\frac{(\frac{3}{4})_n^3(\frac{5}{4})_n}{(1)_{n}^3(\frac{1}{2})_n}
\Big\{H_{n}^{(3)}+H_{n}^{(3)}(-\tfrac{1}{4})\Big\}.
\label{equation-c}
\end{align}
Letting $n\to\infty$ in \eqref{equation-c} and utilizing
\eqref{relation} and \eqref{polygamma-c}, we are led to
\eqref{equation-wei-b}.
\end{proof}

\section{Proof of Theorems \ref{thm-c} and \ref{thm-d}}

In order to prove Theorem \ref{thm-c}, we require Dougall's $_5F_4$
 summation formula (cf. \cite[P. 27]{Bailey}):
\begin{align}
&{_{5}F_{4}}\left[\begin{array}{cccccccc}
  a,1+\frac{a}{2},b,c,d\\
  \frac{a}{2},1+a-b,1+a-c,1+a-d
\end{array};1\right]
\notag\\[1mm]
&\:\:=
\frac{\Gamma(1+a-b)\Gamma(1+a-c)\Gamma(1+a-d)\Gamma(1+a-b-c-d)}{\Gamma(1+a)\Gamma(1+a-b-c)\Gamma(1+a-b-d)\Gamma(1+a-c-d)},
 \label{eq:Dougall}
\end{align}
where  $\mathfrak{R}(1+a-b-c-d)>0$.

Firstly, we begin to prove Theorem \ref{thm-c}.

\begin{proof}[{\bf{Proof of Theorem \ref{thm-c}}}]
A known transformation formula for hypergeometric series (cf.
\cite[Theorem 9]{Chu-b}) can be stated as
\begin{align}
&\sum_{k=0}^{\infty}\frac{(c)_k(d)_k(e)_k(1+a-b-c)_k(1+a-b-d)_{k}(1+a-b-e)_{k}}{(1+a-c)_{k}(1+a-d)_{k}(1+a-e)_{k}(1+2a-b-c-d-e)_{k}}
\notag\\[1mm]
&\quad\times\frac{(-1)^k}{(1+a-b)_{2k}}\alpha_k(a,b,c,d,e) \notag
\end{align}
\begin{align}
&\:=\sum_{k=0}^{\infty}(a+2k)\frac{(b)_k(c)_k(d)_k(e)_k}{(1+a-b)_{k}(1+a-c)_{k}(1+a-d)_{k}(1+a-e)_{k}},
\label{equation-aa}
\end{align}
where  $\mathfrak{R}(1+2a-b-c-d-e)>0$ and
\begin{align*}
\alpha_k(a,b,c,d,e)&=\frac{(1+2a-b-c-d+2k)(a-e+k)}{1+2a-b-c-d-e+k}
\\[1mm]
&\quad+\frac{(1+a-b-c+k)(1+a-b-d+k)(e+k)}{(1+a-b+2k)(1+2a-b-c-d-e+k)}.
\end{align*}
Choosing $b=a$ in \eqref{equation-aa} and evaluating the right
series by \eqref{eq:Dougall}, we discover
\begin{align*}
&\sum_{k=0}^{\infty}\bigg(\frac{-1}{4}\bigg)^k\frac{(c)_k(d)_k(e)_k(1-c)_{k}(1-d)_k(1-e)_{k}}
{(\frac{1}{2})_{k}(1)_{k}(1+a-c)_{k}(1+a-d)_{k}(1+a-e)_{k}(2+a-c-d-e)_{k}}
\notag\\[1mm]
&\quad\times
\bigg\{(1+a-c-d+2k)(a-e+k)+\frac{(1-c+k)(1-d+k)(e+k)}{1+2k}\bigg\}
\notag\\[1mm]
&\:\:=\frac{\Gamma(1+a-c)\Gamma(1+a-d)\Gamma(1+a-e)\Gamma(2+a-c-d-e)}{\Gamma(a)\Gamma(1+a-c-d)\Gamma(1+a-c-e)\Gamma(1+a-d-e)}.
\end{align*}
Perform the substitutions $d\mapsto d-c, e\mapsto e-d$ in the last
equation to arrive at
\begin{align}
&\sum_{k=0}^{\infty}\bigg(\frac{-1}{4}\bigg)^k\frac{(c)_k(d-c)_k(e-d)_k(1-c)_{k}(1+c-d)_k(1+d-e)_{k}\beta_k(a,c,d,e)}
{(\frac{1}{2})_{k}(1)_{k}(1+a-c)_{k}(1+a+c-d)_{k}(1+a+d-e)_{k}(2+a-e)_{k}}
\notag\\[1mm]
&\:\:=\frac{\Gamma(1+a-c)\Gamma(1+a+c-d)\Gamma(1+a+d-e)\Gamma(2+a-e)}{\Gamma(a)\Gamma(1+a-d)\Gamma(1+a+c-e)\Gamma(1+a-c+d-e)},
\label{equation-bb}
\end{align}
where
\begin{align*}
\beta_k(a,c,d,e)&=(1+a-d+2k)(a+d-e+k)
\\[1mm]
&+\frac{(1-c+k)(1+c-d+k)(e-d+k)}{1+2k}.
\end{align*}

Employ the derivative operator  $\mathcal{D}_{c}$ on both sides of
\eqref{equation-bb} to get
\begin{align*}
&\sum_{k=0}^{\infty}\bigg(\frac{-1}{4}\bigg)^k\frac{(c)_k(d-c)_k(e-d)_k(1-c)_{k}(1+c-d)_k(1+d-e)_{k}}
{(\frac{1}{2})_{k}(1)_{k}(1+a-c)_{k}(1+a+c-d)_{k}(1+a+d-e)_{k}(2+a-e)_{k}}
\notag\\[1mm]
&\quad\times
\bigg\{\beta_k(a,c,d,e)\big[H_{k}(c-1)-H_{k}(d-c-1)+H_{k}(c-d)-H_{k}(-c)+H_{k}(a-c)
\notag\\[1mm]
&\qquad -H_{k}(a+c-d)\big]+\frac{(d-2c)(e-d+k)}{1+2k}\bigg\}
\notag\\[1mm]
&\:\:=\frac{\Gamma(1+a-c)\Gamma(1+a+c-d)\Gamma(1+a+d-e)\Gamma(2+a-e)}{\Gamma(a)\Gamma(1+a-d)\Gamma(1+a+c-e)\Gamma(1+a-c+d-e)}
\notag\\[1mm]
&\quad\times
\big\{\psi(1+a+c-d)-\psi(1+a-c)+\psi(1+a-c+d-e)-\psi(1+a+c-e)\big\}.
\end{align*}
Dividing both sides by $d-2c$ and  then fixing
$a=c=e-1=\frac{1}{2}$, we have
\begin{align}
&\sum_{k=0}^{\infty}\bigg(\frac{-1}{4}\bigg)^k\frac{(\frac{1}{2})_k(d-\frac{1}{2})_k^2(\frac{3}{2}-d)_k^2}
{(1)_{k}^3(d)_{k}(2-d)_{k}}\bigg\{\frac{\frac{3}{2}-d+k}{1+2k}+\frac{20k^2+8k-4d^2+8d-3}{8}
\notag\\[1mm]
&\quad\times
\bigg[\sum_{i=1}^k\frac{1}{(-\frac{1}{2}+i)(d-\frac{3}{2}+i)}+\sum_{i=1}^k\frac{1}{(-\frac{1}{2}+i)(\frac{1}{2}-d+i)}
-\sum_{i=1}^k\frac{1}{i(1-d+i)}\bigg]\bigg\}
\notag\\[1mm]
&\:\:=\frac{\Gamma(d)\Gamma(2-d)}{\Gamma(d-\frac{1}{2})\Gamma(\frac{3}{2}-d)}
\frac{\psi(2-d)-\psi(1)+\psi(d-\frac{1}{2})-\psi(\frac{1}{2})}{\pi(d-1)}.
\label{equation-cc}
\end{align}
Applying the derivative operator  $\mathcal{D}_{d}$ on both sides of
\eqref{equation-cc} and then letting $d\to 1$, we catch hold of
\eqref{equation-wei-c} after using  L'H\^{o}pital rule and
\eqref{polygamma-a} and \eqref{polygamma-b}.
\end{proof}

Secondly, we start to prove Theorem \ref{thm-d}.

\begin{proof}[{\bf{Proof of Theorem \ref{thm-d}}}]
A known transformation formula for hypergeometric series (cf.
\cite[Theorem 31]{Chu-b}) can be written as
\begin{align}
&\sum_{k=0}^{\infty}(-1)^k\frac{(b)_k(c)_k(d)_k(e)_k(1+a-b-c)_k(1+a-b-d)_{k}(1+a-b-e)_{k}}{(1+a-b)_{2k}(1+a-c)_{2k}(1+a-d)_{2k}(1+a-e)_{2k}}
\notag\\[1mm]
&\quad\times\frac{(1+a-c-d)_k(1+a-c-e)_{k}(1+a-d-e)_{k}}{(1+2a-b-c-d-e)_{2k}}\mu_k(a,b,c,d,e)
\notag\\[1mm]
&\:=\sum_{k=0}^{\infty}(a+2k)\frac{(b)_k(c)_k(d)_k(e)_k}{(1+a-b)_{k}(1+a-c)_{k}(1+a-d)_{k}(1+a-e)_{k}},
\label{equation-dd}
\end{align}
where $\mathfrak{R}(1+2a-b-c-d-e)>0$ and
\begin{align*}
&\mu_k(a,b,c,d,e)\\[1mm]
&\:=\frac{(1+2a-b-c-d+3k)(a-e+2k)}{1+2a-b-c-d-e+2k}+\frac{(e+k)(1+a-b-c+k)}{(1+a-b+2k)(1+a-d+2k)}
\\[1mm]
&\quad\times\frac{(1+a-b-d+k)(1+a-c-d+k)(2+2a-b-d-e+3k)}{(1+2a-b-c-d-e+2k)(2+2a-b-c-d-e+2k)}
\\[1mm]
&\:+\frac{(c+k)(e+k)(1+a-b-c+k)(1+a-b-d+k)}{(1+a-b+2k)(1+a-c+2k)(1+a-d+2k)(1+a-e+2k)}
\\[1mm]
&\quad\times\frac{(1+a-b-e+k)(1+a-c-d+k)(1+a-d-e+k)}{(1+2a-b-c-d-e+2k)(2+2a-b-c-d-e+2k)}.
\end{align*}
Setting $e=a$ in \eqref{equation-dd} and calculating the right
series by \eqref{eq:Dougall}, we find
\begin{align*}
&\sum_{k=0}^{\infty}(-1)^k\frac{(a)_k(b)_k(c)_k(d)_k(1-b)_k(1-c)_{k}(1-d)_{k}}{(1)_{2k}(1+a-b)_{2k}(1+a-c)_{2k}(1+a-d)_{2k}}
\end{align*}
\begin{align*}
&\quad\times\frac{(1+a-b-c)_k(1+a-b-d)_{k}(1+a-c-d)_{k}}{(2+a-b-c-d)_{2k}}\nu_k(a,b,c,d)
\notag\\[1mm]
&\:=\frac{\Gamma(1+a-b)\Gamma(1+a-c)\Gamma(1+a-d)\Gamma(2+a-b-c-d)}{\Gamma(1+a)\Gamma(1+a-b-c)\Gamma(1+a-b-d)\Gamma(1+a-c-d)},
\end{align*}
where
\begin{align*}
&\nu_k(a,b,c,d)\\[1mm]
&\:=\frac{2k(1+2a-b-c-d+3k)}{a}+\frac{(a+k)(1+a-b-c+k)}{a(1+a-b+2k)}
\\[1mm]
&\quad\times\frac{(1+a-b-d+k)(1+a-c-d+k)(2+a-b-d+3k)}{(1+a-d+2k)(2+a-b-c-d+2k)}
\\[1mm]
&\:+\frac{(a+k)(c+k)(1-b+k)(1-d+k)}{a(1+2k)(1+a-b+2k)(1+a-c+2k)}
\\[1mm]
&\quad\times\frac{(1+a-b-c+k)(1+a-b-d+k)(1+a-c-d+k)}{(1+a-d+2k)(2+a-b-c-d+2k)}.
\end{align*}
Perform the replacements $c\mapsto c-b, d\mapsto d-c$ in the last
equation to obtain
\begin{align}
&\sum_{k=0}^{\infty}(-1)^k\frac{(a)_k(b)_k(c-b)_k(d-c)_k(1-b)_k(1+b-c)_{k}(1+c-d)_{k}}{(1)_{2k}(1+a-b)_{2k}(1+a+b-c)_{2k}(1+a+c-d)_{2k}}
\notag\\[1mm]
&\quad\times\frac{(1+a-c)_k(1+a-b+c-d)_{k}(1+a+b-d)_{k}}{(2+a-d)_{2k}}\theta_k(a,b,c,d)
\notag\\[1mm]
&\:=\frac{\Gamma(1+a-b)\Gamma(1+a+b-c)\Gamma(1+a+c-d)\Gamma(2+a-d)}{\Gamma(1+a)\Gamma(1+a-c)\Gamma(1+a-b+c-d)\Gamma(1+a+b-d)},
\label{equation-ee}
\end{align}
where
\begin{align*}
&\theta_k(a,b,c,d)\\[1mm]
&\:=\frac{2k(1+2a-d+3k)}{a}+\frac{(a+k)(1+a-c+k)}{a(1+a-b+2k)}
\\[1mm]
&\quad\times\frac{(1+a-b+c-d+k)(1+a+b-d+k)(2+a-b+c-d+3k)}{(1+a+c-d+2k)(2+a-d+2k)}
\\[1mm]
&\:+\frac{(a+k)(c-b+k)(1-b+k)(1+c-d+k)}{a(1+2k)(1+a-b+2k)(1+a+b-c+2k)}
\\[1mm]
&\quad\times\frac{(1+a-c+k)(1+a-b+c-d+k)(1+a+b-d+k)}{(1+a+c-d+2k)(2+a-d+2k)}.
\end{align*}

Employ the derivative operator  $\mathcal{D}_{b}$ on both sides of
\eqref{equation-ee} to deduce
\begin{align*}
&\sum_{k=0}^{\infty}(-1)^k\frac{(a)_k(b)_k(c-b)_k(d-c)_k(1-b)_k(1+b-c)_{k}(1+c-d)_{k}}{(1)_{2k}(1+a-b)_{2k}(1+a+b-c)_{2k}(1+a+c-d)_{2k}}
\notag\\[1mm]
&\quad\times\frac{(1+a-c)_k(1+a-b+c-d)_{k}(1+a+b-d)_{k}}{(2+a-d)_{2k}}\times
\big\{\theta_k(a,b,c,d)
\notag\\[1mm]
&\quad\times\big[H_{k}(b-1)-H_{k}(c-b-1)+H_{k}(b-c)-H_{k}(-b)+H_{k}(a+b-d)
\notag\\[1mm]
&\qquad
-H_{k}(a-b+c-d)+H_{2k}(a-b)-H_{2k}(a+b-c)\big]+\mathcal{D}_{b}\theta_k(a,b,c,d)\big\}
\notag\\[1mm]
&\:=\frac{\Gamma(1+a-b)\Gamma(1+a+b-c)\Gamma(1+a+c-d)\Gamma(2+a-d)}{\Gamma(1+a)\Gamma(1+a-c)\Gamma(1+a-b+c-d)\Gamma(1+a+b-d)}
\notag\\[1mm]
&\quad\times
\big\{\psi(1+a+b-c)-\psi(1+a-b)+\psi(1+a-b+c-d)-\psi(1+a+b-d)\big\}.
\end{align*}
Dividing both sides by $c-2b$ and  then taking
$a=b=d-1=\frac{1}{2}$, we conclude
\begin{align}
&\sum_{k=0}^{\infty}(-1)^k\frac{(\frac{1}{2})_k^4(c-\frac{1}{2})_k^3(\frac{3}{2}-c)_k^3}
{(1)_{2k}^3(c)_{2k}(2-c)_{2k}}\bigg\{\lambda_k(c)+\omega_k(c)
\notag\\[1mm]
&\quad\times
\bigg[2\sum_{i=1}^k\frac{1}{(-\frac{1}{2}+i)(c-\frac{3}{2}+i)}-\sum_{i=1}^k\frac{1}{(-\frac{1}{2}+i)(\frac{1}{2}-c+i)}
-\sum_{i=1}^{2k}\frac{1}{i(1-c+i)}\bigg]\bigg\}
\notag\\[1mm]
&\:\:=\frac{2\,\Gamma(c)\Gamma(2-c)}{\Gamma(c-\frac{1}{2})\Gamma(\frac{3}{2}-c)}
\frac{\psi(2-c)-\psi(1)+\psi(c-\frac{1}{2})-\psi(\frac{1}{2})}{\pi(c-1)},
\label{equation-ff}
\end{align}
where
\begin{align*}
\lambda_k(c)&=2k(1+6k)+\frac{(3-2c+2k)(2c-1+2k)(1+2c+6k)}{16(c+2k)}
\\[1mm]
&\quad+\frac{(3-2c+2k)(2c-1+2k)^3}{64(c+2k)(2-c+2k)},
\\[1mm]
 \omega_k(c)&=\frac{(3-2c+2k)^2(19+44c-20c^2+164k+8ck+172k^2)}{64(1+2k)(c+2k)(2-c+2k)^2}.
\end{align*}
Applying the derivative operator  $\mathcal{D}_{c}$ on both sides of
\eqref{equation-ff} and then letting $c\to 1$, we are led to
\eqref{equation-wei-d} after utilizing L'H\^{o}pital rule and
\eqref{polygamma-a} and \eqref{polygamma-b}.
\end{proof}

\textbf{Acknowledgments}\\

The work is supported by  the National Natural Science Foundation of
China (No. 12071103).


\end{document}